\documentclass[12pt]{amsart}
\usepackage{amscd,amssymb}
\usepackage{graphicx}
\usepackage[arrow,matrix]{xy}

\theoremstyle{plain}
\newtheorem{thm}[subsection]{Theorem}

\newtheorem{prop}[subsection]{Proposition}
\newtheorem{cor}[subsection]{Corollary}

\theoremstyle{definition}
\newtheorem{rk}[subsection]{Remark}
\newtheorem{definition}[subsection]{Definition}
\newtheorem{ex}[subsection]{Example}

\numberwithin{equation}{section} 

\newcommand{\N}{\mathbb{N}}

\newcommand{\C}{\mathbb{C}}

\newcommand{\OO}{\mathcal{O}}
\newcommand{\g}{\mathcal{G}}
\newcommand{\id}{\mathcal{D}}
\newcommand{\T}{\mathcal{T}}
\newcommand{\h}{\mathcal{H}}
\newcommand{\m}{\mathcal{M}}

\newcommand{\ir}{\mathcal{R}}
\newcommand{\ia}{\mathcal{A}}
\newcommand{\ik}{\mathcal{K}}

\begin{document}
\title[Determinacy of Determinantal Varieties]{Determinacy of Determinantal Varieties}

\author[I. Ahmed]{Imran Ahmed$^*$}
\thanks{* Partially supported by FAPESP, grant \# 2014/00304-2}
\address{Imran Ahmed, Department of Mathematics, COMSATS Institute of Information Technology,
M.A. Jinnah Campus, Defence Road, off Raiwind Road,
Lahore-Pakistan.}
\email{drimranahmed@ciitlahore.edu.pk}
\author[M.A.S. Ruas]{Maria Aparecida Soares Ruas$^*$}
\address{Maria Aparecida Soares Ruas, Departamento de Matem\'{a}tica, Instituto de Ci\^{e}ncias Matem\'{a}ticas e
de Computa\c{c}\~{a}o, Universidade de S\~{a}o Paulo, Avenida
Trabalhador S\~{a}ocarlense 400, S\~{a}o Carlos-S.P., Brazil.}
\email{maasruas@icmc.usp.br}

\begin{abstract}

A more general class than complete intersection singularities is the
class of determinantal singularities. They are defined by the vanishing
of all the minors of a certain size of a $m\times n$-matrix. In
this note, we consider $\g$-finite determinacy of matrices defining a special class of determinantal varieties. They are called essentially isolated
determinantal singularities (EIDS) and were defined by Ebeling and Gusein-Zade \cite{EG}. In this note, we prove that matrices parametrized by generic homogeneous forms of degree $d$ define EIDS. It follows that $\mathcal{G}$-finite determinacy of matrices hold in general. As a consequence, EIDS of a given type $(m,n,t)$ holds in general.
\end{abstract}

\maketitle

\section{Introduction}

Let $M_{mn}\cong\C^{mn}$ be the space of $m\times n$-matrices with
complex entries. Let $M_{mn}^t\subset M_{mn}$ consists of matrices
of rank less than $t$ with $1\leq t\leq\min\{m,n\}$.
One can prove that $M_{mn}^t$ is an irreducible singular algebraic
variety of codimension $(m-t+1)(n-t+1)$, see \cite{BV}. The singular locus of $M_{mn}^t$ coincides with
$M_{mn}^{t-1}$. The singular locus of $M_{mn}^{t-1}$ is
$M_{mn}^{t-2}$ etc., see \cite{ACGH}. The variety $M_{mn}^t$ can be
written as
$$M_{mn}^t=\bigsqcup_{i=1,\ldots,t}(M_{mn}^i\backslash M_{mn}^{i-1}).$$
This partition is a Whitney stratification of $M_{mn}^t$.

Let $F:U\subset\C^N\to M_{mn}$. For each $x\in U\subset\C^N$,
$F(x)=(f_{ij}(x))$ is a $m\times n$-matrix, the coordinates
$f_{ij}$ are complex analytic functions on $U$. The determinantal
variety $X$ of type $(m,n,t)$ in an open domain $U\subset\C^N$ is
the preimage of the variety $M_{mn}^t$ given by $X=F^{-1}(M_{mn}^t)$
such that codimension of $X$ is $(n-t+1)(m-t+1)$ in $\C^N$. We denote by $f$ the map defined by the $t\times t$ minors of $F$.

The essential isolated determinantal singularities (EIDS) were
defined by Ebeling and Gusein-Zade in \cite{EG}. A germ
$(X,0)\subset(\C^N,0)$ of a determinantal variety of type $(m,n,t)$
has an essentially isolated determinantal singularity (EIDS) at
the origin if $F$ is transverse to all strata $M_{mn}^i\backslash
M_{mn}^{i-1}$ of the stratification of $M_{mn}^t$ in a punctured
neighbourhood of the origin. The singular set of $X$ is given by
$F^{-1}(M_{mn}^{t-1})$. We observe that
if $X$ has isolated singularity at the origin, then
$N\leq(m-t+2)(n-t+2)$, see \cite{EG}.

J.W. Bruce in \cite{B} and A. Fr\"{u}hbis-Kr\"{u}ger in \cite{F} defined the group $\mathcal{G}$ acting on the space of matrices, see Definition \ref{d1}.
M.S. Pereira, in \cite{P1} and \cite{P2}, proved that $F:(\C^N,0)\to M_{mn}$ is $\mathcal{G}$-
finitely determined if and only if the determinantal variety $X^t=F^{-1}(M_{mn}^t)$ is an EIDS for all $1\leq t\leq \min \{m,n\}$.

In Theorem \ref{t1}, we prove that $\mathcal{G}$-finite determinacy holds in general. It follows that almost all determinantal varieties of type $(m,n,t)$
are EIDS, see Corollary \ref{c1}. Our proof relies on two other results of $\mathcal{G}$-finite determinacy: Proposition \ref{l1} in which we extend a criterion for finite determinacy of \cite{BRS} to the case of determinantal variety and Proposition \ref{p1} in which we prove that almost all matrices with homogeneous entries of degree $d$ are $\mathcal{G}$-finitely determined. In Proposition \ref{p3}, we give a method to get examples of homogeneous EIDS of arbitrarily high degree from a special class of linear EIDS.

\section{Preliminaries}

Let $\OO_N$ be the ring of germs at $0$ of analytic functions on $\C^N$ and $\m$
its maximal ideal. Let $M_{mn}(\OO_N)$ be the set of holomorphic map germs $F:(\C^N,0)\to M_{mn}$.
This set can be identified with $\OO_N^{mn}$,
where $\OO_N^{mn}$ is a free $\OO_N$-module of rank $mn$.

Let $\ir$ be the group of analytic diffeomorphism germs in
$(\C^N,0)$ which is given by changes of coordinates in source and
$\h=GL_n(\OO_N)\times GL_m(\OO_N)$, where $Gl_n(\OO_N)$ (respectively $Gl_m(\OO_N)$) is the invertible group of $n\times n$ matrices (respectively $m\times m$ matrices) with entries in $\OO_N$. Let $\g=\ir\ltimes \h$ be the
semi-direct product of $\ir$ and $\h$. One can prove that $\g$ is
a geometric subgroup of $\ik$ (\cite{D} and \cite{P1}).

 We say that two germs $F_1,\,F_2\in M_{mn}(\OO_N)$ are $\g$-equivalent if
 there exist germs of diffeomorphisms $h\in\ir$,
 $R\in GL_n(\OO_N)$ and $L\in GL_m(\OO_N)$ such that $F_1=L^{-1}(h^* F_2)R$.

Given a matrix $F\in M_{mn}(\OO_N)$, let $C_{pq}(F)$ (respectively
$R_{lk}(F)$) be the matrix having the $p$-th column (respectively
the $l$-th row) equal to the $q$-th column of $F$ (respectively the
$k$-th row) with zeros elsewhere. We denote by $J(F)$ the submodule $M_{mn}(\OO_N)$ generated by the matrices of the
form $\frac{\partial F}{\partial x_{\lambda}}$ for $\lambda=1,\ldots,N$.

The $\g$-tangent space to a germ $F$ is given by
$$\T\g F=\m J(F)+\OO_N\{R_{lk}(F), C_{pq}(F)\},$$
where $1\leq l,\,k\leq m$ and $1\leq p,\,q\leq n$, see \cite{B}, \cite{P1} and \cite{P2}.
The extended tangent space is given by
$$\T\g_e F=J(F)+\OO_N \{R_{lk}(F), C_{pq}(F)\}.$$
Moreover, the $\g$-codimension, respectively extended $\g$-codimension, of $F$ are given by
$$d(F,\g)=\dim_{\C}\frac{M_{mn}(\OO_N)}{T\g F}\mbox{   and   }d_e(F,\g)=\dim_{\C}\frac{M_{mn}(\OO_N)}{T\g_e F}.$$



Let $\underline{w}=(w_1,\ldots,w_N)\in \N^N$ be a set of weights.
For any monomial $x^{\alpha}=x_1^{\alpha_1}\ldots x_N^{\alpha_N}$,
we define the $\underline{w}$-filtration of $x^{\alpha}$ by
$$fil(x^{\alpha})=\sum_{i=1}^N w_i\alpha_i.$$
We define a filtration on the ring $\OO_N$ as follows
$$fil(f)=\inf_{\alpha}\{fil(x^{\alpha})\,|\,\frac{\partial^{|\alpha|}f}{\partial x^{\alpha}}(0)\neq 0\}$$
for all germ $f\in\OO_N$, where $|\alpha|=\alpha_1+\ldots+\alpha_N$.
We can extend this notion to the ring of 1-parameter families of
germs on $N$ variables, putting
$$fil(x^{\alpha}t^{\beta})=fil(x^{\alpha}).$$
Given a matrix $F\in M_{mn}(\OO_N)$, we define
$fil(F)=\id=(d_{ij})$, where $d_{ij}=fil(f_{ij})$ for $1\leq i\leq
m$, $1\leq j\leq n$.

A matrix $F\in M_{mn}(\OO_N)$ is called weighted homogeneous of type
$(\id;\underline{w})\in M_{mn}(\N)\times \N^N$ if
$fil(f_{ij})=d_{ij}$ with respect to
$\underline{w}=(w_1,\ldots,w_N)$ such that the following relations
are satisfied
$$d_{ij}-d_{ik}=d_{lj}-d_{lk}\,\,\,\forall\,\,\,1\leq i,\,l\leq m,\,1\leq j,\,k\leq n.$$

We recall now $\g$-finite determinacy of the matrices $F:(\C^N,0)\to M_{mn}(\OO_N)$, see \cite{F}.

\begin{definition}\label{d1}
The matrix $$F:(\C^N,0)\to M_{mn}(\OO_N)$$ is $r$-$\g$-determined if for every $G\in M_{mn}(\OO_N)$, $j^rG(0)=j^rF(0)$, then $F\stackrel{\g}{\sim}G$.
If such $r$ exist, we say that $F$ is $\g$-finitely determined.
\end{definition}

The following theorems were proved in \cite{P1}.

\begin{thm}(\cite{P1}, Th. 2.3.1) \label{tp1}
The following conditions are equivalent.\\
\noindent (i) $F\in M_{mn}(\OO_N)$ is $\g$-finitely determined.\\
\noindent (ii) $T\g F\supset \m_N^k M_{mn}(\OO_N)$ for some positive integer $k$.\\
\noindent (iii) $d_e(F,\g)<\infty$.
\end{thm}

\begin{thm}(\cite{P1}, Th. 2.4.1)\label{tp2}
The following conditions are equivalent.\\
\noindent (i) $F:(\C^N,0)\to M_{mn}$ is $\g$-finitely determined.\\
\noindent (ii) $X^t=F^{-1}(M_{mn}^t)$ is an EIDS for all $1\leq t\leq \min\{m,n\}$.
\end{thm}

For a Cohen-Macaulay codimension $2$ determinantal variety with isolated singularity, the following result was proved in \cite{P1} and \cite{P2}.

\begin{thm}(\cite{P1}, Th. 2.6.1)\label{tp3}
Let $F:(\C^N,0)\to M_{(n+1)n}$ define a Cohen-Macaulay codimension $2$ isolated singular variety $X=F^{-1}(M_{(n+1)n}^n)$. Let $f$ be the map defined by the $n\times n$ minors of $F$. Then, the following conditions are equivalent.\\
\noindent (i) $X\cap V(J(f))=\{0\}$, where $J(f)$ is the ideal generated by the $2\times 2$-minors of the Jacobian matrix of $f$.\\
\noindent (ii) $J(f)+\langle f_1,\ldots,f_{n+1}\rangle\supset \m_N^r$, for some positive integer $r$.\\
\noindent (iii) $F$ is $\g$-finitely determined.
\end{thm}

\section{Main Theorem}

We are motivated by \cite{BRS} in establishing the following
result.

\begin{prop}\label{l1}
We assign positive integer weights $w_k$ to $x_k$, $k=1,\ldots,N$.
Suppose that $f_{ij}(x_1,\ldots,x_N)$ is weighted homogeneous of
degree $d_{ij}$, $1\leq i\leq m,\,1\leq j\leq n$, w.r.t. weights
$(w_1,\ldots,w_N)$. Suppose $F:(\C^N,0)\to M_{mn}$ is finitely
$\g$-determined and weighted homogeneous of type $(\id;\underline{w})$. Let $G:(\C^N,0)\to M_{mn}$ be a polynomial mapping $G=(g_{ij})$ such that $degree \,g_{ij}<d_{ij}$. Then,
$F+G:(\C^N,0)\to M_{mn}$ is finitely $\g$-determined. Moreover,
$$d_e(F,\g)\geq d_e(F+G,\g).$$
\end{prop}

\begin{proof}
Consider $\C^N$ with a fixed coordinate system $x_1,\ldots,x_N$. We
assume that the weighted homogeneity type
$\underline{w}=(w_1,\ldots,w_N)$ is fixed. Take
\[
F= \left(
  \begin{array}{cccc}
    f_{11}(x) & f_{12}(x)& \ldots & f_{1n}(x) \\
        \vdots & \vdots &  & \vdots \\
    f_{m1}(x) & f_{m2}(x) & \ldots & f_{mn}(x) \\
  \end{array}
\right),
\]
where $f_{ij}(x)$ is weighted homogeneous of degree $d_{ij}$, $1\leq
i\leq m,\,1\leq j\leq n$, w.r.t. weights $(w_1,\ldots,w_N)$. Let
$h:(\C^N,0)\to(\C^N,0)$ be a map and $s\in\C^*$. We define the
deformation $h_s:(\C^N,0)\to(\C^N,0)$ of $h$ by
$$h_s(x_1,\ldots,x_N)=(s^{-w_1}x_1,\ldots,s^{-w_N}x_N),\,\,s\in\C^*.$$
Then,
\[
h_s^*(F)= \left(
  \begin{array}{cccc}
    s^{-d_{11}}f_{11}(x) & s^{-d_{12}}f_{12}(x)& \ldots & s^{-d_{1n}}f_{1n}(x) \\
        \vdots & \vdots &  & \vdots \\
    s^{-d_{m1}}f_{m1}(x) & s^{-d_{m2}}f_{m2}(x) & \ldots & s^{-d_{mn}}f_{mn}(x) \\
  \end{array}
\right).
\]
Let
\begin{center}
$\psi_s= \left(
  \begin{array}{cccc}
    s^{a_{11}} & 0 & \ldots & 0 \\
    \vdots & \vdots &      & \vdots \\
    0 & s^{a_{22}} & \ldots & 0 \\
        \vdots & \vdots & \ddots & \vdots \\
    0 & 0 & \ldots & s^{a_{mm}} \\
  \end{array}
\right)$ and $\phi_s= \left(
  \begin{array}{cccc}
    s^{b_{11}} & 0 & \ldots & 0 \\
    \vdots & \vdots &      & \vdots \\
    0 & s^{b_{22}} & \ldots & 0 \\
        \vdots & \vdots & \ddots & \vdots \\
    0 & 0 & \ldots & s^{b_{nn}} \\
  \end{array}
\right)$
\end{center}
Now, we compute
\[
\psi_s h_s^*(F)\phi_s= \left(
  \begin{array}{cccc}
    s^{a_{11}-d_{11}+b_{11}}f_{11}(x) & s^{a_{11}-d_{12}+b_{22}}f_{12}(x)& \ldots & s^{a_{11}-d_{1n}+b_{nn}}f_{1n}(x) \\
    \vdots & \vdots &      & \vdots \\
    s^{a_{22}-d_{21}+b_{11}}f_{21}(x) & s^{a_{22}-d_{22}+b_{22}}f_{22}(x) & \ldots & s^{a_{22}-d_{2n}+b_{nn}}f_{2n}(x) \\
        \vdots & \vdots & \ddots & \vdots \\
    s^{a_{mm}-d_{m1}+b_{11}}f_{m1}(x) & s^{a_{mm}-d_{m2}+b_{22}}f_{m2}(x) & \ldots & s^{a_{mm}-d_{mn}+b_{nn}}f_{mn}(x) \\
  \end{array}
\right)
\]
$\,\,\,\,\,\,\,\,\,\,\,\,\,\,\,\,\,\,\,\,\,\,\,\,\,\,\,\,\,=(s^{a_{ii}-d_{ij}+b_{jj}}f_{ij}(x))$\\
We want to find conditions on $\psi_s$ and $\phi_s$ such that $\psi_s h_s^*(F)\phi_s=F$.
Since $F$ is a weighted homogeneous polynomial map of filtration
$\id$, therefore, by definition,
$$d_{ij}-d_{ik}=d_{lj}-d_{lk}\,\,\,\forall\,\,\,1\leq i,\,l\leq m,\,1\leq j,\,k\leq n.$$
Then, the relations
$$a_{ii}-d_{ij}+b_{jj}=0,\,a_{ii}-d_{il}+b_{ll}=0,\,a_{kk}-d_{kj}+b_{jj}=0,\,a_{kk}-d_{kl}+b_{ll}=0$$
hold if
$$b_{jj}-b_{ll}=d_{ij}-d_{il},\,b_{jj}-b_{ll}=d_{kj}-d_{kl}$$
and
$$a_{ii}-a_{kk}=d_{ij}-d_{kj},\,a_{ii}-a_{kk}=d_{il}-d_{kl}.$$

By hypothesis, $F$ is finitely $\g$-determined, therefore we get that
$$F_s=\psi_s h_s^*(F+G)\phi_s=F+\psi_s h_s^*(G)\phi_s$$
is also finitely $\g$-determined for sufficiently small $s$. But
$F_s$ is $\g$-equivalent to $F+G$, so $F+G$ is finitely
$\g$-determined. Moreover, by the upper semicontinuity of the codimension,
$$d_e(F,\g)\geq d_e(F_s,\g).$$
\end{proof}

\begin{prop}\label{p1}
 Let $L_{(a,\lambda)}(x)$ be the $m\times n$ matrix whose entries are
 $$l_{(a(i,j),\lambda)}(x)=\sum_{\lambda=1}^{N}a(i,j)_{\lambda}x_{\lambda}^d$$
 homogeneous polynomials of degree $d\in \mathbb{N}^*$ with $1\leq i\leq m$; $1\leq j\leq n$, $1\leq\lambda\leq N$.
 Then, for almost all values of the parameters $a=(a(i,j)_{\lambda})$, $L_{(a,\lambda)}:(\C^N,0)\to M_{mn}$ is $\g$-finitely determined.
\end{prop}

\begin{proof}
For any $d\in \mathbb{N}^*$, and for all $(i,j)$, $1\leq i\leq m$; $1\leq j\leq n$, $1\leq\lambda\leq N$,
 let $l_{(a(i,j),\lambda)}(x)$ denote the following homogeneous polynomial of degree $d$:
 $$l_{(a(i,j),\lambda)}(x)=\sum_{\lambda=1}^{N}a(i,j)_{\lambda}x_{\lambda}^d.$$
 Let $\mathcal{A}$ be the space of parameters $a=(a(i,j)_{\lambda})$, $1\leq i\leq m$, $1\leq j\leq n$ and $1\leq \lambda\leq N$. It is a
vector space of dimension $Nmn$. Let
$$\mathcal{L}:\mathbb{C}^N\times\mathcal{A}\to M_{mn},$$
$$(x,a)\mapsto \mathcal{L}(x,a)=L_{(a,\lambda)}(x)=$$
\[
\left(
  \begin{array}{cccc}
    a_{11}^1 x_1^d+\ldots+a_{11}^N x_N^d & \ldots & a_{1n}^1 x_1^d+\ldots+a_{1n}^N x_N^d \\
    \vdots &     & \vdots \\
    a_{m1}^1 x_1^d+\ldots+a_{m1}^N x_N^d & \ldots & a_{mn}^1 x_1^d+\ldots+a_{mn}^N x_N^d \\
  \end{array}
\right).
\]
So the entries of the $m\times n$ matrix
$L_{(a,\lambda)}(x)$ are $$l_{(a(i,j),\lambda)}(x)=\sum_{\lambda=1}^N a(i,j)_{\lambda}x_{\lambda}^d.$$
The partial derivative $\frac{\partial \mathcal{L}}{\partial a(i,j)_{\lambda}}=x_{\lambda}^d$.
Then, for all $x\neq 0$, the map $\mathcal{L}$ is a submersion. So, for every submanifold
$\mathcal{S}\subset M_{mn}$, we can apply the Basic Transversality Theorem to get that for almost all
$a\in\mathcal{A}$, $\mathcal{L}_a:\mathbb{C}^N\to M_{mn}$ is transverse to $\mathcal{S}$, see \cite{GG}. In particular,
for $\mathcal{S}=S^i=M_{mn}^i\backslash M_{mn}^{i-1}$ and the result follows.
\end{proof}

\subsection{$\g$-Determinacy holding in general}

As an application of Proposition \ref{l1} and Proposition \ref{p1}, we show that the property of finite $\g$-determinacy holds in general in $M_{mn}(\OO_N)$. Following Wall \cite{W}, Sec. 5, we introduce first the notion of a property holding in general for map germs $f:(\C^n,0)\to(\C^p,0)$.

If $A_r\subset J^r(n,p)$ are algebraic sets with $A_{r+1}\subset\pi^{-1}(A_r)$, where $\pi:J^{r+1}(n,p)\to J^r(n,p)$ is the canonical projection, then the set $A$ of map germs $f:(\C^n,0)\to(\C^p,0)$ with $J^r(f)\in A_r$ for all $r$ is said to be {\it proalgebraic}. We know that the codimension of $A_r$ is less than or equal to the codimension of $A_{r+1}$. We write $codim\,A=\lim_{r\to\infty}\,codim\,A_r$. A property of map germs which holds for all except those in a proalgebraic set of infinite codimension is said to {\it hold in general}. The codimension of the proalgebraic set $A$ is infinity if and only if, for every $r$-jet $z\in J^r(n,p)$, we can find $f\notin A$ with $j^r(f)=z$.

Now, we establish that $\g$-determinacy is a property holding in general in $M_{mn}(\OO_N)$.

\begin{thm}\label{t1}
Finite $\g$-determinacy holds in general in $M_{mn}(\OO_N)$.
\end{thm}

\begin{proof}
Let $G(x)=(g_{ij}(x))$ be a $m\times n$ matrix, where $g_{ij}$ are polynomials of degree $\leq r$.
Following Wall \cite{W}, Sec. 5, to prove that $\g$- finite determinacy holds in general in $M_{mn}(\OO_N)$, we need to find $m\times n$-matrix $H(x)=(h_{ij}(x))$ such that $j^r(H)=G$, and $H$ is  $\g$-finitely determined. By Proposition \ref{p1}, we can take $F(x)$ with homogeneous entries of degree $d\geq r+1$  and $\g$-finitely determined. Then, we take $H=G+F$. So, $j^r(H)=G$ and applying Proposition \ref{l1} we get result.
\end{proof}

The following corollary is a direct consequence of above result.

\begin{cor}\label{c1}
Almost all determinantal varieties of type $(m,n,t)$ with $1\leq t\leq \min\{m,n\}$ are EIDS.
\end{cor}

\section{Examples}

Linear determinantal varieties have been intensively studied by many authors, see for instance \cite{DRS}, \cite{EH} and \cite{RS}.
We discuss here a method to get examples of homogeneous EIDS of arbitrarily high degree from a special class of linear EIDS.

Let $F:\C^N\to M_{mn}$ be given by
\[
F(x)=
\left(
  \begin{array}{cccc}
    a_{11}^1x_1^d+\ldots+a_{11}^Nx_N^d & \ldots & a_{1n}^1x_1^d+\ldots+a_{1n}^Nx_N^d \\
    \vdots &     & \vdots\\
    a_{m1}^1x_1^d+\ldots+a_{m1}^Nx_N^d & \ldots & a_{mn}^1x_1^d+\ldots+a_{mn}^Nx_N^d  \\
  \end{array}
\right).
\]
We can associate to $F(x)$ a $m\times n$ linear matrix with variables $u_{\lambda}=x_{\lambda}^d$ for $\lambda=1,\ldots,N$, given by
\[
L(u)=
\left(
  \begin{array}{cccc}
    a_{11}^1u_1+\ldots+a_{11}^Nu_N & \ldots & a_{1n}^1u_1+\ldots+a_{1n}^Nu_N \\
    \vdots &     & \vdots\\
    a_{m1}^1u_1+\ldots+a_{m1}^Nu_N & \ldots & a_{mn}^1u_1+\ldots+a_{mn}^Nu_N  \\
  \end{array}
\right).
\]
We denote by $F_{\lambda}$ and $L_{\lambda}$ the matrices $\frac{\partial F}{\partial x_{\lambda}}$ and $\frac{\partial L}{\partial u_{\lambda}}$ (respectively), $\lambda=1,\ldots,N$.

We establish now the following result.

\begin{prop}\label{p3}
With the above notations, we suppose that the following condition holds for the linear EIDS $L$:
\begin{equation}\label{eq1}
\OO_N\{u_{\lambda} L_{\lambda}(u),R_{lk}(u),C_{pq}(u)\}\supseteq m_N^r\OO_N^{mn},
\end{equation}
for some natural number $r$. Then, $F$ is an EIDS.
\end{prop}

\begin{proof}
The $\g_e$-tangent space of $F$ is the $\OO_N$ module generated by $\{F_{\lambda}(x), R_{lk}(x), C_{pq}(x)\}$, $\lambda=1,\ldots, N$. By the chain rule, we have that $F_{\lambda}(x)=d x_{\lambda}^{d-1}L_{\lambda}(u)$. It implies that $x_{\lambda} F_{\lambda}(x)=d u_{\lambda} L_{\lambda}(u)$. Notice that the matrices $R_{lk}$ and $C_{pq}$ of $F$ are obtained from corresponding ones of $L$ just by replacing $u_{\lambda}$ by $x_{\lambda}^d$.

Now, by hypothesis $L$ satisfies condition (\ref{eq1}). Then, if $E_{ij}$ denotes the $m\times n$ matrix with $1$ in the $ij$-entry and zeros elsewhere, then condition (\ref{eq1}) implies
\begin{equation}\label{eq2}
u_{\lambda}^r E_{ij}\in \OO_N\{u_{\lambda} L_{\lambda}, R_{lk}, C_{pq}\},
\end{equation}
for all $1\leq \lambda\leq N$ and $1\leq i\leq m$, $1\leq j\leq n$, for some $r\in \mathbb{N}^*$. It implies that
\begin{equation}\label{eq3}
x_{\lambda}^{dr}E_{ij}\in T\g F.
\end{equation}
As $T\g F$ is an $\OO_N$-module, the condition (\ref{eq3}) will imply that $T\g F\supseteq m_N^s\OO_N^{mn}$ for some natural number $s=s(r)$.
\end{proof}

\begin{rk}
Condition (\ref{eq1}) is stronger than the condition that $L$ defines an EIDS. For example,
$L(u)=\left(
  \begin{array}{cccc}
    u_1 & u_2 \\
    u_3 & u_4  \\
  \end{array}
\right)$
defines an EIDS in $\C^4$ because the matrices $\{L_{\lambda}\}$, $\lambda=1,\ldots,4$, generate the whole space $\OO_N^{mn}$. However condition (\ref{eq1}) does not hold for $L$. The corresponding matrix
$
F(x)=
\left(
  \begin{array}{cccc}
    x_1^d & x_2^d \\
    x_3^d & x_4^d  \\
  \end{array}
\right)$
does not define an EIDS.
\end{rk}

In the following examples, we give determinantal varieties defining EIDS.

\begin{ex}
The matrix $F:(\C^4,0)\to M_{22}$ is given by
\[
F(x)=
\left(
  \begin{array}{cccc}
    x_1^d-x_2^d & x_3^d+x_4^d \\
    x_3^d-x_4^d & x_1^d+x_2^d  \\
  \end{array}
\right)
=
\left(
  \begin{array}{cccc}
    u_1-u_2 & u_3+u_4 \\
    u_3-u_4 & u_1+u_2  \\
  \end{array}
\right),
\]
where $x_i^d=u_i$ with $i=1,\ldots,4$.
One can see that the variety $X=F^{-1}(M_{22}^2)$ has isolated singularities and therefore by Theorem \ref{tp3} is finitely determined. However we discuss here how to apply Proposition \ref{p3} to obtain the result. To prove that $X$ defines an EIDS, it is sufficient to prove that $L(u)$ satisfies (\ref{eq1}). The generators are\\
$\langle
\left(
  \begin{array}{cccc}
    u_1 & 0 \\
      0 & u_1 \\
  \end{array}
\right),
\left(
  \begin{array}{cccc}
    -u_2 & 0 \\
    0 & u_2 \\
  \end{array}
\right),
\left(
  \begin{array}{cccc}
    0 & u_3 \\
    u_3 & 0 \\
  \end{array}
\right),
\left(
  \begin{array}{cccc}
    0 & u_4 \\
    -u_4 & 0 \\
  \end{array}
\right),\\
\left(
  \begin{array}{cccc}
   u_1-u_2 & 0 \\
    u_3-u_4 & 0 \\
  \end{array}
\right),
\left(
  \begin{array}{cccc}
    u_3+u_4 & 0 \\
    u_1+u_2 & 0 \\
  \end{array}
\right),
\left(
  \begin{array}{cccc}
    0 & u_1-u_2 \\
    0 & u_3-u_4 \\
  \end{array}
\right),
\left(
  \begin{array}{cccc}
   0 & u_3+u_4 \\
    0 & u_1+u_2 \\
  \end{array}
\right),\\
\left(
  \begin{array}{cccc}
    u_1-u_2 & u_3+u_4 \\
    0 & 0 \\
  \end{array}
\right),
\left(
  \begin{array}{cccc}
   u_3-u_4 & u_1+u_2 \\
    0 & 0 \\
  \end{array}
\right),\\
\left(
  \begin{array}{cccc}
    0 & 0 \\
    u_1-u_2 & u_3+u_4 \\
  \end{array}
\right),
\left(
  \begin{array}{cccc}
    0 & 0 \\
    u_3-u_4 & u_1+u_2 \\
  \end{array}
\right)
\rangle$.\\
Let
$e_1=\left(
  \begin{array}{cccc}
    1 & 0 \\
    0 & 0 \\
  \end{array}
\right),\,
e_2=\left(
  \begin{array}{cccc}
    0 & 1 \\
    0 & 0 \\
  \end{array}
\right),\,
e_3=\left(
  \begin{array}{cccc}
    0 & 0 \\
    1 & 0 \\
  \end{array}
\right)$ and
$e_4=\left(
  \begin{array}{cccc}
    0 & 0 \\
    0 & 1 \\
  \end{array}
\right).$\\

Then, we have the following system of equations for $j=1,\ldots,4$:\\
$u_1u_je_1+u_1u_je_4=0,\,\,\,-u_2u_je_1+u_2u_je_4=0;$\\
$u_3u_je_2+u_3u_je_3=0,\,\,\,u_4u_je_2-u_4u_je_3=0;$\\
$(u_1u_j-u_2u_j)e_1+(u_3u_j-u_4u_j)e_3=0,\,(u_3u_j+u_4u_j)e_1+(u_1u_j+u_2u_j)e_3=0;$\\
$(u_1u_j-u_2u_j)e_2+(u_3u_j-u_4u_j)e_4=0,\,(u_3u_j+u_4u_j)e_2+(u_1u_j+u_2u_j)e_4=0;$\\
$(u_1u_j-u_2u_j)e_1+(u_3u_j+u_4u_j)e_2=0,\,(u_3u_j-u_4u_j)e_1+(u_1u_j+u_2u_j)e_2=0;$\\
$(u_1u_j-u_2u_j)e_3+(u_3u_j+u_4u_j)e_4=0,\,(u_3u_j-u_4u_j)e_3+(u_1u_j+u_2u_j)e_4=0.$\\

Now, we can use MATLAB to see that condition (\ref{eq1}) is satisfied for $L$. We then use Proposition \ref{p3} to get that $F$ defines an EIDS.
\end{ex}

\begin{ex}
Let $F:(\C^N,0)\to M_{2N}$ be given by
\[
F(x)=
\left(
  \begin{array}{cccc}
    x_1^d & x_2^d & \ldots & x_N^d \\
    x_N^d & x_1^d & \ldots & x_{N-1}^d  \\
  \end{array}
\right)
\,\mbox{ and }\,
L(u)=\left(
  \begin{array}{cccc}
    u_1 & u_2 & \ldots & u_N \\
    u_N & u_1 & \ldots & u_{N-1}  \\
  \end{array}
\right),
\]
where $x_i^d=u_i$ with $i=1,\ldots,N$. The variety $X=F^{-1}(M_{2N}^2)$ has codimension $N-1$, so $X$ is a determinantal curve in $\C^N$.
To prove that $X$ defines an EIDS, it is sufficient to prove that $L(u)$ satisfies (\ref{eq1}). The generators of (\ref{eq1}) are\\
$$\{B_s,R_{lk},C_{pq}:\,1\leq s,p,q\leq N;\,1\leq l,k\leq 2\},$$
where $B_s=(b_{ij})$, $\,i=1,\,2; \,1\leq j\leq N$ with $b_{1s}=u_s$; $b_{1j}=0$ for $j\neq s$ and $b_{2(s+1)}=u_s$; $b_{2j}=0$ for $j\neq s+1$ such that $b_{2(N+1)}=b_{21}=u_N$.

We denote by $e_r$ the $2\times N$-matrix, where $r=1,\ldots,2N$, having  $1$ in $r$-th position and $0$ elsewhere. Then, we have the following system of equations for $j=1,\ldots,N$ in $(N+N^2+4)N$ equations and $N^2(N+1)$ variables:\\
$u_1u_je_1+u_1u_je_{N+2}=0,\,u_2u_je_2+u_2u_je_{N+3}=0,\ldots, u_Nu_je_N+u_Nu_je_{N+1}=0;$\\
$u_mu_je_m+u_{m-1}u_je_{N+m}=0$ for $1\leq m\leq N$ such that $u_0=u_N$;\\
$\sum_{i=1}^{N}u_iu_je_i=0,\,u_Nu_je_1+\sum_{i=1}^{N-1}u_iu_je_{i+1}=0$;\\
$\sum_{i=1}^{N}u_iu_je_{N+i}=0,\,u_Nu_je_{N+1}+\sum_{i=1}^{N-1}u_iu_je_{i+1+N}=0.$\\

Now, we can use MATLAB to see that the matrix of coefficients is reduced to the matrix $\left(
  \begin{array}{cccc}
    I_{N^2(N+1)} \\
    O  \\
  \end{array}
\right)$. Therefore, condition (\ref{eq1}) is satisfied for $L$. We then use Proposition \ref{p3} to get that $F$ defines an EIDS.
\end{ex}

{\bf Acknowledgement}: Authors thank  A. Simis and H. Pedersen for the fruitful discussions on determinantal varieties.

\end{document}